\RequirePackage{ifpdf}
\ifpdf 
\documentclass[pdftex]{sigma}
\else
\documentclass{sigma}
\fi

\begin{document}

\allowdisplaybreaks
	
\renewcommand{\PaperNumber}{106}

\FirstPageHeading

\renewcommand{\thefootnote}{$\star$}

\ShortArticleName{Biorthogonal Expansion of Non-Symmetric Jack Functions}

\ArticleName{Biorthogonal Expansion\\ of Non-Symmetric Jack Functions\footnote{This paper is a
contribution to the Proceedings of the 2007 Midwest
Geometry Conference in honor of Thomas~P.\ Branson. The full collection is available at
\href{http://www.emis.de/journals/SIGMA/MGC2007.html}{http://www.emis.de/journals/SIGMA/MGC2007.html}}}

\Author{Siddhartha SAHI~$^\dag$ and Genkai ZHANG~$^\ddag$}

\AuthorNameForHeading{S.~Sahi  and G.~Zhang}

\Address{$^\dag$~Department of Mathematics, Rutgers University, New Brunswick, New Jersey, USA}
\EmailD{\href{mailto:sahi@math.rutgers.edu}{sahi@math.rutgers.edu}}

\Address{$^\ddag$~Mathematical Sciences,
Chalmers University of Technology and  Mathematical Sciences,\\
$\phantom{^\ddag}$~G\"oteborg University, Sweden}
\EmailD{\href{mailto:genkai@math.chalmers.se}{genkai@math.chalmers.se}}

\ArticleDates{Received August 08, 2007, in f\/inal form October 31, 2007; Published online November 15, 2007}

\Abstract{We f\/ind a biorthogonal expansion
of the Cayley transform of the non-symmetric Jack
functions in terms of the non-symmetric Jack polynomials,
the coef\/f\/icients being Meixner--Pollaczek type polynomials. This
is done by computing the Cherednik--Opdam transform
of the non-symmetric Jack polynomials multiplied
by the exponential function.}


\Classification{33C52; 33C67; 43A90}


\def\del{\delta}
\def\Del{\Delta}
\def\fpw{\frac{\partial\,}{\partial w}}
\def\fa{\mathfrak a}
\def\fb{\mathfrak b}
\def\fg{\mathfrak g}
\def\fk{\mathfrak k}
\def\fh{\mathfrak h}
\def\fl{\mathfrak l}
\def\fm{\mathfrak m}
\def\fn{\mathfrak n}
\def\fp{\mathfrak p}
\def\fq{\mathfrak q}
\def\fs{\mathfrak s}
\def\ft{\mathfrak t}
\def\fv{\mathfrak v}
\def\fz{\mathfrak z}

\def\fac{\mathfrak a^{\mathbb C}}
\def\fbc{\mathfrak b^{\mathbb C}}
\def\fgc{\mathfrak g^{\mathbb C}}
\def\fkc{\mathfrak k^{\mathbb C}}
\def\fhc{\mathfrak h^{\mathbb C}}
\def\flc{\mathfrak l^{\mathbb C}}
\def\fmc{\mathfrak m^{\mathbb C}}
\def\fnc{\mathfrak n^{\mathbb C}}
\def\fpc{\mathfrak p^{\mathbb C}}
\def\fqc{\mathfrak q^{\mathbb C}}
\def\fsc{\mathfrak s^{\mathbb C}}
\def\ftc{\mathfrak t^{\mathbb C}}
\def\fvc{\mathfrak v^{\mathbb C}}
\def\fzc{\mathfrak z^{\mathbb C}}

\def\mf{\mathfrak}
\def\mfa{\mathfrak a}
\def\mfb{\mathfrak b}
\def\mfg{\mathfrak g}
\def\mfk{\mathfrak k}
\def\mfh{\mathfrak h}
\def\mfl{\mathfrak l}
\def\mfm{\mathfrak m}
\def\mfn{\mathfrak n}
\def\mfp{\mathfrak p}
\def\mfP{\mathfrak P}
\def\mfq{\mathfrak q}
\def\mfs{\mathfrak s}
\def\mfv{\mathfrak v}
\def\mfz{\mathfrak z}
\def\a{\alpha}
\def\b{\beta}
\def\epsi{\epsilon}
\def\Claminv2{|C(\Lambda)|^{-2}}
\def\ga{\gamma}
\def\Ga{\Gamma}
\def\varepsi{\varepsilon}
\def\lam{\lambda}

\def\blam{\underline{\bold \lambda}}
\def\brho{\underline{\bold \rho}}

\def\ba{\underline{\bold \a}}

\def\Lam{\Lambda}
\def\ome{\omega}
\def\Ome{\Omega}
\def\dmua{d\mu_\a}
\def\dmub{d\mu_\beta}
\def\dmuabp{d\mu_{\a+\beta+p}}
\def\de{d\varepsilon}
\def\K#1#2#3{K^{#1}(#2,#3)}
\def\bK#1#2#3{\bar K^{#1}(#2,#3)}
\def\LtvDmua{L^2(D,d\mu_\a)}
\def\U#1#2{U_{#1}^{(#2)}}
\def\bU#1#2{\bar U_{#1}^{(#2)}}
\def\Aa2D{A^{\a,2}(D)}
\def\bAa2D{\overline{A^{\a,2}(D)}}
\def\Ab2D{A^{\beta,2}(D)}
\def\bAb2D{\overline{A^{\beta,2}(D)}}
\def\AtA{\Aa2D \otimes \bAb2D}
\def\abs#1{\vert#1\vert}
\def\innerp#1#2{(#1, #2)}
\def\Innerp#1#2{\left(#1, #2\right)}
\def\ainnerp#1#2{\langle#1, #2\rangle}
\def\norm#1{\Vert#1\Vert}
\def\Norm#1_#2{\Vert#1\Vert_{#2}}
\def\Normsq#1{\Vert#1\Vert^2}
\def\normsq#1{\vert#1\vert^2}
\def\pd#1#2{\frac{\partial #1}{\partial #2}}
\def\2pd#1#2{\frac{\partial^2 #1}{\partial #2^2}}
\def\p11d#1#2#3{\frac{\partial^2 #1}{  \partial #2\partial #3  }}
\def\der#1#2{\frac{d #1}{d #2}}
\def\b{\beta}
\def\ga{\gamma}
\def\Claminv2{|C(\Lambda)|^{-2}}
\def\del{\delta}
\def\Del{\Delta}
\def\bdd{bounded}
\def\symdom{ symmetric domain}
\def\varepsi{\varepsilon}
\def\sig{\sigma}
\def\Sig{\Sigma}
\def\Ga{\Gamma}
\def\lam{\lambda}
\def\Lam{\Lambda}
\def\ad{\operatorname{ad}}
\def\bod{\bold}
\def\ch{\operatorname{ch}}
\def\coth{\operatorname{coth}}
\def\curl{\operatorname{curl}}
\def\det{\operatorname{det}}
\def\dive{\operatorname{div}}
\def\grad{\operatorname{grad}}
\def\tanh{\operatorname{tanh}}
\def\coth{\operatorname{coth}}
\def\ch{\operatorname{ch}}
\def\exp{\operatorname{exp}}
\def\Exp{\operatorname{Exp}}
\def\sh{\operatorname{sh}}
\def\tanh{\operatorname{tanh}}
\def\sh{\operatorname{sh}}
\def\sgn{\operatorname{sgn}}
\def\Ker{\operatorname{Ker}}
\def\Res{\operatorname{Res}}
\def\Tr{\operatorname{Tr}}
\def\dmua{d\mu_\a}
\def\dmub{d\mu_\beta}
\def\dmuabp{d\mu_{\a+\beta+p}}
\def\de{d\varepsilon}
\def\K#1#2#3{K^{#1}(#2,#3)}
\def\bK#1#2#3{\bar K^{#1}(#2,#3)}
\def\LtvDmua{L^2(D,d\mu_\a)}
\def\U#1#2{U_{#1}^{(#2)}}
\def\bU#1#2{\bar U_{#1}^{(#2)}}
\def\Aa2D{A^{\a,2}(D)}
\def\bAa2D{\overline{A^{\a,2}(D)}}
\def\Ab2D{A^{\beta,2}(D)}
\def\bAb2D{\overline{A^{\beta,2}(D)}}
\def\AtA{\Aa2D \otimes \bAb2D}

\def\m{\underline{\mathbf m}}
\def\n{\underline{\mathbf n}}
\def\l{\underline{\mathbf l}}
\def\ub1#1{\underline{\mathbf 1^{#1}}}

\def\ov{\overline{v}}
\def\s{\underline{\bold s}}
\def\u{\underline{\bold u}}

\def\ov{\overline{v}}
\def\intR{\int_{\mathbb R^2}}
\def\intC{\int_{\mathbb C^2}}

\def\za{\aa{}}
\def\zae{\"a{}}
\def\zo{\"o{}}
\def\Za{\AA{}}
\def\Zae{\"A{}}
\def\Zo{\"O{}}
\def\bc{\mathbb C}
\def\br{\mathbb R}
\def\bn{\mathbb N} 
\def\nat0{\mathbb Z_{\ge 0}} 
\def\bz{\mathbb Z}
\def\bq{\mathbb Q}

\def\mM{\mathcal M}
\def\mL{\mathcal L}

\def\bc{\mathbb C}
\def\br{\mathbb R}
\def\bn{\mathbb N}
\def\bz{\mathbb Z}
\def\bq{\mathbb Q}

\def\draft{\centerline{(Draft {\the \day}/{\the\month} \the \year.)}}

\def\phila{\phi_{\blam}}
\def\vepm{\varepsilon_{\m, \nu}(\blam)}
\def\vepmp{\varepsilon_{\m^\prime, \nu}(\blam)}
\def\kam{\kappa_{\m, \nu}(\blam)}
\def\nutnu{\pi({\nu})\otimes \overline{\pi(\nu)}}
\def\pia2ta2{\pi({\frac a2})\otimes \overline{\pi(\frac a2)}}
\def\Hnu{\mathcal H_{\nu}}
\def\Fnu{\mathcal F_{\nu}}
\def\Pm{\mathcal P_{\m}}
\def\PV{\mathcal P(V_{\bc})}
\def\Pa{\mathcal P(\mathfrak a)}
\def\Pn{\mathcal P_{\n}}
\def\cL#1nu{\mathcal L_{{#1}, \nu}}
\def\cL#1a2{\mathcal L_{{#1}, \frac a2}}
\def\E#1nu{E_{{#1}, \nu}}
\def\E#1a2{E_{{#1}, \frac a2}}

\def\Dc{D_{\mathbb C}}
\def\Vc{V_{\mathbb C}}
\def\Bnu{B_{\nu}} 
\def\el{e_{\blam}} 
\def\bnul{b_{\nu}(\lam)} 
\def\GGa{\Gamma_{\Ome}} 
\def\fc#1#2{\frac{#1}{#2}} 
\def\SS{\mathcal S}

\def\hth{\mathcal H_{\nu_1} \otimes \mathcal H_{\nu_2}}


\section{Introduction}

In \cite{Opdam-acta} Opdam studied the non-symmetric eigenfunctions of the Cherednik operators associated to a root system with general multiplicity and proved the Plancherel formula for the correspon\-ding ``Cherednik--Opdam'' transform. For the root system of type $A$ the polynomial eigenfunctions are also called the non-symmetric Jack polynomials and they have been extensively studied (see e.g.~\cite{sahi-newprod}).
There are other related non-symmetric polynomials such as
the Laguerre polynomials which are the eigenfunctions
of  the Hankel transform, which is basically the Fourier transform
on the underlying space. The non-symmetric Laguerre polynomials
 form an orthogo\-nal basis for the $L^2$-space and it is thus a natural problem to f\/ind their Cherednik--Opdam transforms.
 In this paper we prove that they are, apart from a factor of Gamma functions, the non-symmetric Meixner--Pollaczek (MP) polynomials and we f\/ind a formula for them in terms of binomial coef\/f\/icients. As a corollary we f\/ind in Theorem~\ref{theorem:4.3} a biorthogonal expansion of the Cayley transform of the non-symmetric Jack functions in terms of the non-symmetric Jack polynomials, the coef\/f\/icients being the MP polynomials. In
 the one variable case the Jack function is just the power function $x^{i\lam}$,
and Theorem~\ref{theorem:4.3} gives an expansion of the
function $(1-t^2)^{\frac b2}(\frac{1-t}{1+t})^{i\lam}$
of the MP polynomials  (\ref{eq:mp-1}); see \cite[(1.7.11)]{Askey-Scheme}.

There are basically three important families of polynomials associated with the root system of type A, namely, the Jack type polynomials, the Laguerre polynomials, and the MP type polynomials that are orthogonal with respect the Harish-Chandra measure $|c(\lam)|^{-2}d\lam$ multiplied with a certain Gamma factor, which is the Heckman--Opdam transform of an exponential function.
Our results give a somewhat unif\/ied picture of the relation between these polynomials and provide a combinatorial formula for the MP type polynomials. In brief, the Laplace transform maps the Laguerre polynomials into Jack polynomials, and the Cherednik--Opdam transform maps the Laguerre polynomials into MP type polynomials.
In the case when the root multiplicities correspond to that of a symmetric cone some results of this type have been obtained in \cite{FK-book, gz-br2} and \cite{doz}.

\section[Non-symmetric Jack functions and the Opdam-Cherednik transform]{Non-symmetric Jack functions\\ and the Opdam--Cherednik transform}

In this section we recall the def\/inition of the non-symmetric Jack polynomials and functions and the Plancherel formula for the Opdam--Cherednik transform, developed in \cite{Opdam-acta}.

We consider the root system of type $A_{r-1}$ in $\mathbb R^r$.
For the purpose of studying Laplace transform we make a change of variables $x_j =e^{2t_j}$ where $t=(t_1, \dots, t_r)\in \mathbb R^r$, and consider functions on $x\in \mathbb R_+^r$ instead.
We f\/ix an ordering of the roots so that (with some abuse of notation) the positive roots are $x_2 -x_1, x_3-x_2, \dots, x_{r}-x_{r-1}$ with root multiplicity $a:=\frac 2{\alpha}$ and we will identify the roots
as vectors in $\mathbb R^r$. Let $\rho$ be the half sum of positive roots, so that
$\rho=
(\rho_1, \rho_2, \dots, \rho_r)=\frac 1{\alpha}(-r+1, -r +3, \dots, r-1)$.
 We consider the measure
\[
d\mu(x)=
\frac{1}{2^r}(x_1 \cdots x_r)^{-\frac 1{\alpha}(r-1) -1}
\prod_{1\le j<k\le r}|x_j-x_k|^{a}
dx_1\cdots dx_r
\]
on $\mathbb R_+^r$
and the corresponding Hilbert space $L^2(\mathbb R_+^r, d\mu)$.

We consider the Dunkl operators
\begin{gather*}
T_j=\partial_j + \frac 1{\alpha} \sum_{i\ne j}\frac{1}{x_j-x_i}
(1-s_{ij})
\end{gather*}
and the Cherednik operators
\begin{gather*}
U_j=U_j^{A}=x_j\partial_j+ \frac 1{\alpha} \sum_{i<j}\frac{x_j}{x_j-x_i}
(1-s_{ij})
+\frac 1{\alpha} \sum_{j<k}\frac{x_k}{x_j-x_k}
(1-s_{jk}) -\frac 12 \rho_j.
\end{gather*}
 Here $\partial_j=\frac{\partial}{\partial x_j}$ and $s_{ij}$ stands for the permutation $(ij)$ that acts on functions $f(x_1, \dots, x_r)$ by interchanging the variables $x_i$ and $x_j$. The operators $\{U_j\}$ can be expressed in terms of $\{T_j\}$ and the multiplication operators $\{x_j\}$, but we will not need this here.
 The Dunkl operators $\{T_j\}$ commute with each other, as do the Cherednik operators $\{U_j\}$.

The polynomial eigenfunctions of the operators $\{U_j\}$ are the non-symmetric Jack polynomials $E_{\eta}(x)=E(\eta, x)$, with
$\eta=(\eta_1, \dots, \eta_r)\in \mathbb N^r$.
They are characterized as the unique eigen-polynomials
of $\{U_j\}$ with leading coef\/f\/icients
$x^{\eta}=x_1^{\eta_1}\cdots
x_r^{\eta_r}$ in the sense that
\[
 E_{\eta}=x^{\eta} +\sum_{\zeta < \eta} c_{\eta \zeta} x^{\zeta}.
\]
We recall that $\zeta < \eta$ here stands for the partial ordering def\/ined by
\[
\zeta < \eta \qquad{\text{if\/f}}\quad
\begin{cases}\zeta^{+} < \eta^{+}, &{\zeta^{+} \ne \eta^{+}},\vspace{1mm}\\
\zeta < \eta, &       \zeta^{+} = \eta^{+},
\end{cases}
\]
where $\eta^{+}$ is the unique partition obtained by permuting the entries of $\zeta$ and $ <$ stands for the natural dominance ordering:
$\zeta < \eta$ if\/f $\sum_{j=1}^p(\zeta_j-\eta_j)\ge 0$, $1\le p\le r$.

The function $E_{\eta}$ has holomorphic extension in the variable $\eta$. More precisely, there exists a function $G_{\lam}(x)=G(\lam, x)$
which we call
the non-symmetric
Jack function,  real analytic in $x\in \mathbb R^r_+$ and holomorphic in $\lam\in {\mathbb C}^r$, such that $G({\lam}, 1^r)=1$, with $1^r=(1, \dots, 1)$, and
\[
U_j G(\lam, x)=\lam_j G(\lam, x).
\]
The relation between $G_{\lam}( x)$ and $E_{\eta}(x)$ is
\begin{gather}
\label{G-E}
G_{\eta +\rho}(x)=\mathcal E_{\eta}(x):= \frac{E_{\eta}(x)}{E_{\eta}(1^r)}.
\end{gather}
 The value
 ${E_{\eta}(1^r)}$ has been computed by Sahi~\cite{sahi-newprod},
\[
{E_{\eta}(1^r)} =\frac{e_{\eta}}
{d_{\eta}},
\]
where $e_{\eta}$ and ${d_{\eta}}$ are def\/ined in the next section. (See also \cite{Opdam-acta} for general root systems.)
Def\/ine
\begin{gather*}
\mathcal F^w [f](\lambda)
=\int_{\mathbb R_+^r} f(x)G(-\lam,  w^{-1}x)d\mu(x).
\end{gather*}
Then writing $d\omega(\lam)$ for the Euclidean measure on the positive Weyl chamber $(\mathbb R^r)_+=\{\lam; \lam_1 >\cdots >\lam_r\}$, we have
\begin{gather}
\label{planc}
\int_{\mathbb R_+^r}|f(x)|^2 d\mu(x)=
\sum_{w \in S_r}\int_{i(\mathbb R^r)_+}
\mathcal F^w [f](\lambda)
\overline{\mathcal F^w (f)(\lambda)}d\tilde\mu(\lam);
\end{gather}
the Plancherel measure $d\tilde\mu$ is given by
\[
d\tilde\mu(\lam)=\frac{(2\pi)^{-r}\tilde c_{w_0}^2(\rho(k), k)}
{\tilde c(\lam)
c(w_0 \lam)}d\omega(\lam),
\]
where $w_0\in S_r$ is the longest Weyl group element, and $c(\lam)$, $\tilde c_w(\lam)$ are the Harish-Chandra $c$-functions \cite{Opdam-acta}.

\section{Laplace transform}

Adapting the notation in \cite{Baker-Forrester-Duke} we denote $q:=1+ \frac 1{\alpha}(r-1)$.
Recall  that the non-symmetric Laplace transform is def\/ined by
\[
\tilde{\mathcal L}
[f](t)=\int_{[0, \infty)^r}
\mathcal K_{A}(-t, x)f(x)
d\mu(x),
\]
where
\[
\mathcal K_{A}(t, y)=
\sum_{\eta}\a^{|\eta|}
\frac{1}{d_{\eta}^\prime }
\mathcal E_{\eta}(t)E_{\eta}(y)
\]
is the non-symmetric analogue of the hypergeometric ${}_0 F_0$ function. It satisf\/ies $\mathcal K_{A}(cx; y)=\mathcal K_{A}(x; cy)$ and  $\mathcal K_{A}(wx; y)= \mathcal K_{A}(x; wy)$ for any $w\in S_r$.

Here each tuple $\eta$ will be identif\/ied with a diagram of nodes $s=(i, j)$, $1\le j\le \eta_i$, $d_\eta =\prod_{s\in \eta}d(s)$,
$d_\eta^\prime =\prod_{s\in \eta}d^\prime(s)$,
$e_\eta^\prime =\prod_{s\in \eta}e(s)$
with
\[
d^\prime (s)=\a(a(s)+1) +l(s), \qquad
d(s)=d^\prime (s)+1,\qquad e(s)=\a(a^\prime(s)+1)+r-l^\prime(s))
\]
and $a(s)=\eta_i-j$, $a^\prime(s)=j-1
$ being the arm length, arm colength and
\begin{gather*}
l(s)=\#\{k>i: j\le \eta_k\le \eta_i\} +
\#\{k<i: j\le \eta_k+1\le \eta_i\},
\\
l'(s)=\#\{k>i:  \eta_k> \eta_i\} +
\#\{k<i:  \eta_k\le \eta_i\},
\end{gather*}
the leg length and leg colength, which were def\/ined in \cite{sahi-newprod} and \cite{Sahi-Knop-inv}.

\begin{remark}
Note that the Laplace transform $\tilde{\mathcal L}$ def\/ined here dif\/fers from the Laplace transform~${\mathcal L}$ def\/ined in
\cite{Baker-Forrester-Duke} by a factor $(\prod_{j=1}^r
x_j)^{-q}$ in the integration. More precisely our measure $d\mu$ is $(\prod_{j=1}^r x_j)^{-q}$ times the  f\/inite measure
$\prod_{1\le j<k\le r}|x_j-x_k|^{a}
dx_1\cdots dx_r$ used in \cite[(3.67)]{Baker-Forrester-Duke}, so that
\[
\widetilde{\mathcal L}[f(x)]
={\mathcal L}\Big[f(x)\Big(\prod_{j=1}^r
x_j\Big)^{-q}\Big].
\]
The advantage of $\widetilde{\mathcal L}$ is that there
is no shift of $q$ in Lemma~\ref{lemma3.3} below, and that in the case of symmetric spaces of type A the measure $d\mu$ is precisely the radial part  of the invariant measure~\cite{FK-book}.
\end{remark}

The function $\mathcal K_{A}(x, y)$ generalizes the exponential function in the sense that
\begin{gather}
\label{T-K}
T_i^{(x)}
\mathcal K_{A}(x, y)= y_i \mathcal K_{A}(x, y),
\end{gather}
where $T_i^{(x)}$ is the Dunkl operator acting on the variable
$x$.

Recall further the def\/inition of generalized Gamma function
\[
\Gamma_{\a}(\kappa)=\prod_{j=1}^r\Gamma\left(\kappa_j
-\frac 1{\alpha}(j-1)\right)
\]
and the Pochammer symbol
\[
[\nu]_{\kappa}=\frac{\Gamma_{\a}(\nu+\kappa)}{\Gamma_{\a}(\nu)}.
\]
for $\nu, \kappa\in \mathbb C^r$, whenever it makes sense.
A scalar  $c\in \mathbb C$ will also be identif\/ied with $(c, \dots, c)\in \mathbb C$ in the text below.  We will also use the abbreviation $x^c =x_1^c \cdots x_r^c$ and $1+x =(1+x_1,\dots, 1+ x_r)$ etc.

We recall further the binomial coef\/f\/icients $\binom{\eta }{\nu}$
for $\eta, \nu \in \mathbb N^r$ are def\/ined by the expansion
\begin{gather}\label{bin-1}
\mathcal E_{\eta}(1+t)=\sum_{\nu}\binom{\eta}{\nu} \mathcal E_{\nu}(t).
\end{gather}
See also \cite{sahi-newprod} and \cite{Sahi-binom}. We make the following generalization.
\begin{definition} \label{definition3.1} The binomial coef\/f\/icients $\binom{\eta}{\nu}$ for any
$\eta \in \mathbb C^r$ and
$\nu \in \mathbb N^r$ are def\/ined by
\begin{gather*}
 G_{\eta +\rho}(1+t)=\sum_{\nu \in \mathbb N^r}
\binom{\eta}{\nu} E_{\nu}(t).
\end{gather*}
\end{definition}
Since $ G_{\lambda}(1+t)$ is an analytic function near $1$
and $E_{\nu}(t)$ form a basis for all polynomials the above def\/inition makes sense, and it agrees with \eqref{bin-1} in view of the relation (\ref{G-E}). In particular $\binom{\eta}{\nu}$ is a polynomial of $\eta\in \mathbb C^r$. It follows from the def\/inition and \cite[Proposition 3.18]{Baker-Forrester-Duke}
that
\begin{gather*}
E_{\nu}(T) G_{\eta +\rho}(t)|_{t=1} =
\frac{d^{\prime}_{\nu}}{\alpha^{|\nu|}}\binom{\eta}{\nu}.
\end{gather*}
Here $E_{\nu}(T)$ stands for the dif\/ferential-dif\/ference
operator obtained from
$E_{\nu}(x)$  replacing $x_j$  by~$T_j$.
We will need a slight generalization of the binomial coef\/f\/icient:
If $w\in S_r$ we def\/ine $\binom{\eta}{\nu}_w$ by
\begin{gather}
\label{E-T-G-w}
E_{\nu}(T) G_{\eta +\rho}(wt)|_{t=1} =
\frac{d^{\prime}_{\nu}}{\alpha^{|\nu|}}\binom{\eta}{\nu}_w.
\end{gather}

The following lemma is proved in \cite[(4.38)]{Baker-Forrester-Duke}.
(Note that there is a typo there: $E_{{\eta}}^{(L)}(\frac 1x)$
 should be replaced by $E_{{\eta}}(\frac 1x)$. For symmetric case this was a conjecture of Macdonald \cite{Mac-leidennotes} proved by e.g.~in \cite[(6.1)--(6.3)]{BF-cmp97}.)

\begin{lemma} \label{lemma3.3} Suppose $c> q-1$. The Laplace transform of the functions
$x^c E_{\eta}(x)$ is given by
\[
\widetilde{\mathcal L}[x^c \mathcal E_{\eta}(x)](t)
=\mathcal N_0^{(L)} [c]_{\eta} \left(\prod_{j=1}^r t_j^{-c}\right) \mathcal E_{{\eta}}\left(\frac 1 t\right).
\]
\end{lemma}
The normalization constant
 $\mathcal N_0^{(L)}$ (depending on $c$) is computed in
\cite{Baker-Forrester-Duke}.

We f\/ix in the rest of the text $b>q-1$. For simplicity we assume also that $b$ is an even integer.
Let
\begin{gather*}
E_{\kappa}^{(L)}(x)=
E_{\kappa}^{(L, b)}(x)=
\frac{(-1)^{|\kappa|}[b]_{\kappa} e_{\kappa}}{d_{\kappa}}
\sum_{\sig}
\frac{(-1)^{|\sig|}   } {[b]_{\sig}}
\binom{\kappa}{\sig}
\mathcal E_{\sig}(x)
\end{gather*}
be the non-symmetric Laguerre polynomial.
The next lemma follows from \cite[Proposition 4.35]{Baker-Forrester-Duke} after a change of variable $2x=t^2$. Our parameter
$b$ is their $a+q$.
\begin{lemma}  Suppose $b>(q-1)=\frac 1{\alpha} (r-1)$.
 The Laguerre functions
\[
l_{\kappa}(x):=l_{\kappa}^{b}(x):=
E_{\kappa}^{(L)}(2x) e^{-p_1(x)}(2x)^{\frac b2}, \qquad p_1(x):=\sum_{j=1}^r x_j,
\]
form an orthogonal basis for the space $L^2(\mathbb R^r_+, d\mu)$.
\end{lemma}

The norm $\Vert l_{\kappa}\Vert^2$ has also been explicitly evaluated in \cite{Baker-Forrester-Duke}.

We can now compute the Cherednik--Opdam transform of the Laguerre
functions. For
 $\eta=(\eta_1, \eta_{2}, \dots, \eta_r)$, denote $\eta^*=(\eta_r, \eta_{r-1}, \dots, \eta_1)$.

\begin{proposition} \label{proposition3.5} Suppose $b>(q-1)=\frac 1{\alpha} (r-1)$.
The Cherednik--Opdam transform
$\mathcal F^w [l_{\kappa}]$, $w\in S_r$,
  of the Laguerre function
$l_{\kappa}(x)
$, is
\[
\mathcal F^w [l_{\kappa}](\lambda)
=
{2^{\frac {rb}2}
\mathcal N_0^{(L)} }
\frac
{\Gamma_{\alpha}(\frac b2-\rho-\lam )}
{\Gamma_{\alpha}(\frac b2)}
M_{\kappa}^w(\lambda),
\]
where
\[
M_{\kappa}^w(\lambda)=
\frac{(-1)^{|\kappa|}[b]_{\kappa}
e_{\kappa}}{d_{\kappa}}
\sum_{\sig}
\frac{1 } {[b]_{\sig}}
\frac{d_{\sig}^\prime d_{\sig}}
{\alpha^{|\sig|} e_{\sig} }
(-2)^{|\sig|}
\binom{\kappa}{\sig}
\binom{-\frac b2 +\lam^\ast +\rho^*}{\sig}_w.
\]
\end{proposition}

\begin{proof} The function $l_{\kappa}$ is a linear combinations
of  the functions $e^{-p_1(x)}E_{\sig}(x)$ and we  compute
the Cherednik--Opdam transform of these functions.
We write Lemma~\ref{lemma3.3} as
\[
\int_{\mathbb R_+^r} \mathcal K_A(-t, x)
 x^{\frac  b2} \mathcal E_{\eta}(x) d\mu(x)
=\mathcal N_0^{(L)}  \frac{\Gamma_{\alpha}(\frac b2+\eta )}
{\Gamma_{\alpha}(\frac b2)}
 t^{-\frac  b2}
\mathcal E_{-{\eta}^\ast}(t)=
\mathcal N_0^{(L)}  \frac{\Gamma_{\alpha}(\frac b2+\eta )}
{\Gamma_{\alpha}(\frac b2)}
\mathcal E_{-\frac b2
-{\eta}^\ast}(t).
\]
Replacing $t$ by $wt$, and using $K(wt, x)=K(t, wx)$
\cite[Theorem 3.8]{Baker-Forrester-Duke}, we f\/ind
\[
\int_{\mathbb R_+^r} \mathcal K_A(-t, x)
 (w^{-1}x)^{\frac  b2} \mathcal E_{\eta}(w^{-1}x) d\mu(x)
=\mathcal N_0^{(L)}  \frac{\Gamma_{\alpha}(\frac b2+\eta )}
{\Gamma_{\alpha}(\frac b2)}
\mathcal E_{-\frac b2
-{\eta}^\ast}(wt).
\]
Here we have used the relations $y^{c}\mathcal E_{\sig}({y}) =\mathcal E_{c+\sig}( {y})$. Let the operator $E_{{\sigma}}(T)$ act on
it  evaluated as $t=1^r=(1, \dots, 1)$.
The resulting equality, by (\ref{T-K}) and the fact that $\mathcal K_A(x, 1^r)=e^{p_1(x)}$, and~(\ref{E-T-G-w}), is,
\[
\int_{\mathbb R_+^r}
e^{-p_1(x)}\mathcal E_{{\sigma}}(-x)
 (w^{-1}x)^{\frac  b2} \mathcal E_{\eta}(w^{-1}x)d\mu(x)
=\mathcal N_0^{(L)}
\frac{\Gamma_{\alpha}(\frac b2+\eta )}
{\Gamma_{\alpha}(\frac b2)}
\frac{d_{\sig}^\prime d_{\sig}}
{\alpha^{|\sig|} e_{\sig} }\binom{-\frac b2
-{\eta}^\ast}{\sig}_w.
\]
We write  $\eta=-\lam -\rho$. Using the relation (\ref{G-E}) we see that
\begin{gather*}
\mathcal F^w [l_{\kappa}](\lambda)
=\int_{[0, \infty)^r}
l_{\kappa}(x) G(-{\lam}, x)d\mu(x)\\
\qquad{} =2^{\frac {rb}2}\mathcal N_0^{(L)}
\frac{\Gamma_{\alpha}(\frac b2 -\rho -\lam )}
{\Gamma_{\alpha}(\frac b2)}
\frac{(-1)^{|\kappa|}[b]_{\kappa}
e_{\kappa}}{d_{\kappa}}
\sum_{\sig}
\frac{1 } {[b]_{\sig}}
\frac{d_{\sig}^\prime d_{\sig}}
{\alpha^{|\sig|} e_{\sig} }
(-2)^{|\sig|}
\binom{\kappa}{\sig}
\binom{-\frac b2 +\lam^\ast +\rho^*}{\sig}_w.
\end{gather*}
as claimed. Our result for general $\lam$ follows
by analytic continuation.
\end{proof}

The Plancherel formula
\eqref{planc} then gives an orthogonality relation for the polynomials
$M_{\kappa}^w(\lam)$:
\[
\sum_{w} \int_{\lam_1 >\cdots >\lam_r }
M_{\kappa}^w(\lam) \overline{M_{\kappa^\prime}^w(\lam)}
|\Gamma_{\alpha}
( \frac b2 -\rho -\lam  )|^2 d\widetilde{\mu}(\lam)
=C \delta_{\kappa, \kappa^\prime} \Vert l_{\kappa}\Vert^2
\]
for some constant $C$ independent of $\kappa$, $\kappa^\prime$.

 In the next section we will f\/ind
another formula for $M_{\kappa}(\lambda)$.

\begin{remark} In the case of one variable $r=1$ we have
$E_l=x^l$,
 $d_l=d'_le_l = l!$, our polynomial
is
\[
M_k(\lam)=(b)_{k} (-1)^k
\sum_{l}
\frac{1} {(b)_{l}}
l!(-2)^{l}
\binom{k}{l}
\binom{-\frac b2 +\lam }{l}
=(b)_{k} (-1)^k
{}_2 F_1\left(-k, \frac b2 -\lam; b; 2\right)
\]
which is the MP polynomial  \cite[(1.7.1)]{Askey-Scheme}
$P_k^{(b/2)}(x; \frac \pi 2)$, more precisely
\begin{gather}
  \label{eq:mp-1}
M_k(\lam)=k! i^k P_k^{(b/2)}\left(-i\lam; \frac \pi 2\right).
\end{gather}
The functions $\Gamma(\frac b2 +i\lam) M_k(i\lam)$
are orthogonal in the space $L^2(\mathbb R, d\lam)$,
as a consequence of \eqref{planc}.
\end{remark}

\section[A binomial formula for $M_{\kappa}(\lambda)$]{A binomial formula for $\boldsymbol{M_{\kappa}(\lambda)}$}

We recall f\/irst an expansion of the kernel $\mathcal K_{A}$
in \cite[Proposition 4.13]{Baker-Forrester-Duke}.
We adopt the shorthand notation $\frac{z}{1-z}=
\prod_{j=1}^r \frac{z_j}{1-z_j}$, in particular
 $\frac{1-z}{1+z}=\prod_{j=1}^r \frac{1-z_j}{1+z_j}$ will be called the Cayley transform of $z=(z_1, \dots, z_r)$.

\begin{lemma} \label{lemma4.1} The following expansion holds
\[
(1-z)^{-b}\mathcal K_{A}\left(-x; \frac{z}{1-z}\right)
=\sum_{\eta}(-\a)^{|\eta|}
\frac{1} {d_{\eta}^\prime } E_{\eta}^{(L)}(x)
\mathcal E_{\eta}(z).
\]
\end{lemma}

We need another expansion of the
$ (1-z)^{-b} \mathcal E_{\kappa }(\frac{1-z}{1+z})$
in terms of the polynomials $E_{\eta}(z)$.

\begin{lemma} \label{lemma4.2} Consider the following expansion
\[
(1-z)^{-b}
\mathcal E_{\eta }\left( \frac{1-z}{1+z}\right)
=\sum_{\kappa} \mathcal C_{\kappa}(\eta)
 \mathcal E_{\kappa}(z),
\]
for $\eta\in \mathbb N^r$, $\eta_j\ge \frac b2$.
The coefficients are then given by
\begin{gather*}
\mathcal C_{\kappa}(\eta)=
\mathcal E_{\eta-b}( -1)
\sum_{\sig} (-2)^{|\sig|}\binom{\eta-b}{\sig}
\binom{-{\sig}^\ast-b}{\kappa}
\end{gather*}
and is a polynomial in $\eta$.
\end{lemma}
\begin{proof} Note that it follows from the remark after Def\/inition~\ref{definition3.1}
that $\mathcal C_{\kappa}(\eta)$ is a polynomial in~$\eta$.
Change variables $y_j=1+z_j$, $\frac{1-z_j}{1+z_j}=\frac{2}{y_j}-1$,
${1-z_j^2}=(1+z_j)^2\frac{1-z_j}{1+z_j}=y_j^2(\frac{2}{y_j}-1)$.
We have
\begin{gather}
(1-z)^{-b}
\mathcal E_{\eta }\left( \frac{1-z}{1+z}\right) =(1-z^2)^{-\frac{b}2}
\mathcal E_{\eta -\frac b2}\left(\frac{1-z}{1+z}\right)
=
\left(y^2 \left(\frac{2}{y}-1\right)\right)^{-\frac b 2}
\mathcal  E_{\eta -\frac b2}\left(\frac{2}{y}-1\right)\nonumber\\
\phantom{(1-z)^{-b} \mathcal E_{\eta }\left( \frac{1-z}{1+z}\right)}{}
= y^{-b}
  \mathcal E_{\eta- b}\left( \frac{2}{y}-1\right)=
(-1)^{|\eta- b|}
y^{-b} \mathcal E_{\eta-b}\left( 1-\frac{2}{y}\right).\label{32}
\end{gather}
We expand $\mathcal E_{\eta-b}( 1-\frac{2}{y})$ by using the binomial formula,
\begin{gather*}
 y^{-b}
\mathcal E_{\eta-b} \left(1-\frac{2}{y}\right)
= y^{-b}
\sum_{\sig} \binom{\eta-b}{\sig}
\mathcal E_{\sig}\left( -\frac{2}{y}\right)
=
\sum_{\sig}
\binom{\eta-b}{\sig}
(-2)^{|\sig|}
\mathcal E_{-\sig^\ast -b}( y).
\end{gather*}
Here we have used the relations
$\mathcal E_{\sig}( -\frac{2}{y}) =
(-2)^{|\sig|}\mathcal E_{-\sig^*}( {y})
$ and $y^{-b}
E_{-\sig^*}( {y}) =E_{-\sig^*-b}( {y})$; see~\cite{Baker-Forrester-Duke}.
Now each term $\mathcal E_{-{\sig}^*-b}({y})
=\mathcal E_{-{\sig}^*-b}( 1+z)
$ can again
be expanded in terms of
$\mathcal E_{\kappa}( z)$.
Inter\-changing the order of summation we f\/ind that (\ref{32}) is
\[
\mathcal E_{\eta-b}( -1)
\sum_{\kappa}\left(
\sum_{\sig} \binom{\eta-b}{\sig}
(-2)^{|\sig|}
\binom{-{\sig}^\ast-b}{\kappa}\right)
\mathcal E_{\kappa}(z).
\]
This completes the proof.
\end{proof}

In the statement of Lemma~\ref{lemma4.2} we have written $(-1)^{|\eta-b|}$ as $\mathcal E_{\eta-b}( -1)$ since the latter has analytic continuation in $\eta$, while on the other hand $(-1)^{|\eta-b|} =e^{i\pi \sum_{j=1}^r (\eta_j-b)}$ is already analytic in $\eta$).

\begin{theorem} \label{theorem:4.3} The coefficients $\mathcal C_{\kappa}$ in Lemma~{\rm \ref{lemma4.2}} are given by
\[
 \mathcal C_{\kappa}(-\lambda-\rho)
=(-\a)^{|\kappa|}
\frac{1} {d_{\kappa}^\prime }
M_{\kappa}(\lam).
\]
Thus we have an expansion
\[
(1-z)^{-b}
G\left(-\lam, \frac{1-z}{1+z}\right)
=\sum_{\kappa} (-\a)^{|\kappa|}
\frac{1} {d_{\kappa}^\prime }
M_{\kappa}(\lam)
 \mathcal E_{\kappa}(z).
\]
\end{theorem}

\begin{proof} We replace $x$ by $2x$ in Lemma \ref{lemma4.1},
\[
\prod_{j=1}^r(1-z_j)^{-b}\mathcal K_{A}\left(-x; \frac{2z}{1-z}\right)
=\sum_{\kappa}(-\a)^{|\kappa|}
\frac{1} {d_{\kappa}^\prime} E_{\kappa}^{(L)}(2x)
\mathcal E_{\kappa}(z).
\]
Multiplying both sides
by $\prod_{j=1}^r (2x_j)^{\frac b 2}
e^{-p_1(x)}$ and since
$\mathcal K_{A}(-x; y+1^r)=
\mathcal K_{A}(-x; y)e^{-p_1(x)}$
we get
\[
\prod_{j=1}^r(1-z_j)^{-b}\mathcal K_{A}\left(-x; \frac{1+z}{1-z}\right)
\prod_{j=1}^r{(2x_j)^{\frac b 2}}
=\sum_{\kappa}(-\a)^{|\kappa|}
\frac{1} {d_{\kappa}^\prime
}l_{\kappa}^{(\nu)}(x)
\mathcal E_{\kappa}(z).
\]
We f\/ix an $\eta \in \mathbb N^r$, written also as $\eta=-\lam -\rho$,
and let $z$ be in a small neiborhood of $0$.
Integrating both sides against $\mathcal E_{\eta} (x)d\mu(x)$ (we omit the routine estimates guaranteeing that the interchanging of  the integration and summation is valid), we get, by Lemma~\ref{lemma3.3} (and Proposition~\ref{proposition3.5}),
\[
\mathcal N_0^{(L)}\left[\frac b2\right]_{\eta}
\prod_{j=1}^r(1-z_j)^{-b}
\mathcal E_{\eta}\left(\frac{1-z}{1+z}\right)
=\mathcal N_0^{(L)}\left[\frac b2\right]_{\eta}
\sum_{\kappa}(-\a)^{|\kappa|}
M_{\kappa}(\lam)
\frac{1} {d_{\kappa}^\prime }
\mathcal E_{\kappa}(z).
\]
Cancelling the common factor $\mathcal N_0^{(L)}[\frac b2]_{\eta}$ we get
\[
(1-z)^{-b}
 G\left(-\lam, \frac{1-z}{1+z}\right)
=
(1-z)^{-b}
\mathcal E_{\eta }\left(\frac{1-z}{1+z}\right)
=\sum_{\kappa}(-\a)^{|\kappa|}
M_{\kappa}(\eta +\rho)
\frac{1} {d_{\kappa}^\prime }
\mathcal E_{\kappa}(z).
\]
Comparing this with Lemma~\ref{lemma4.2} proves our claim for $\lam=-(\eta +\rho)$. For a general $\lam$ we observe that
$G(-\lam, \frac{1-z}{1+z})$ is an analytic function of $z$ in a neighborhood of $z=0$ and $\{\mathcal E_{\kappa}(z)\}$ for a basis for the polynomials, thus it has a power series expansion, with coef\/f\/icients analytic in~$\lam$, which are in turn determined by their restriction to the values of the form $\lam =-(\eta+\rho)$.
\end{proof}

Note that we may def\/ine a Hilbert space (similar to the Fock space or Bergman space construction e.g.~\cite{Opdam-acta}) so that the polynomials $\mathcal E_{\kappa}(z)$ form an orthogonal basis. However the function $(1-z)^{-b} G(-\lam, \frac{1-z}{1+z})$ is not in the Hilbert space and the above formula is not an expansion in the Hilbert space sense. However it might be interesting to f\/ind a formula for
\[
\sum_{\kappa}   t^{|\kappa|} (-\a)^{|\kappa|}M_{\kappa}(\eta +\rho) \frac{1} {d_{\kappa}^\prime } \mathcal E_{\kappa}(z),
 \]
which is convergent in the Hilbert space for small $t$. This would be then an analogue of Mahler's formula for the Hermite functions.

Finally we can also consider the same problem for symmetric Jack functions, which are the Heckman--Opdam spherical function for a root system of Type A.
The Laplace transform in the symmetric case has been studied earlier by Macdonald~\cite{Mac-leidennotes} and Baker--Forrester~\cite{BF-cmp97}.

For a partition $\kappa=(\kappa_1, \dots, \kappa_r)$, let $\Ome_{\kappa}(x) =\Ome_{\kappa}^{(\alpha)}(x)$ be the Jack symmetric polynomials in $r$-variables, normalized so that $\Ome_{\kappa}(1^r)=1$; see~\cite{Macd-book}. The corresponding function for $\kappa\in \mathbb C^r$, which we call the symmetric
Jack function (and which for general root system is called the Heckman--Opdam hypergeometric function) will be denoted still by $\Ome_{\kappa}(x)$,  $x\in (0, \infty)^r$.

\begin{proposition} Consider the following expansion
\[
\prod_{j=1}^r(1-z_j^2)^{-\frac b 2}\Ome_{\eta -\frac b2}
\left(\frac{1-z}{1+z}\right)
=\sum_{\kappa} Q_{\kappa}(\eta) \Ome_{\kappa}(z).
\]
The coefficient $Q_{\kappa}(\eta)$ are symmetric polynomials in $\eta$
up to a $\rho$ shift, and after a slight modification
\[
f_{\kappa}(\lam)= 2^{rb/2}\mathcal N_0^{(L)}
\frac{d_{\kappa}^\prime }{(-\a)^{|\kappa|}}
Q_{\kappa}(-\lam -\rho)
\]
they form an orthogonal basis in the space
\[
L^2\left((0, \infty)^r,
\left|\frac{\Gamma_{\alpha}\left(\frac b2-\rho -\lam\right)}
{\Gamma_{\alpha}\left(\frac b2\right)}
\right|^2
|c(\lam)|^{-2}ds\right)^{S_r}
\]
of symmetric $L^2$-functions.
\end{proposition}

\begin{proof} The proof is identical to that for the non-symmetric case. Up to a constant, the function $\Gamma_{\alpha}(\frac b2 -\rho -\lam) Q_{\kappa}(-\lam-\rho)$ is the Heckman--Opdam transform of the Laguerre function $e^{p_1(x)}L_{\kappa}(2x)(2x)^{b/2}$, where
$L_{\kappa}(2x)(2x)^{b/2}$ is the symmetric Laguerre polynomial def\/ined in \cite{BF-cmp97}. This follows from the Laplace transform of the symmetric Jack polynomials studied in \cite{BF-cmp97}. Thus the orthogonality of the functions $f_{\kappa}$ is a consequence of the Plancherel formula \eqref{planc} and the orthogonality of the Laguerre functions \cite{BF-cmp97}.
\end{proof}

We observe the Heckman--Opdam transform of the function $L_{\kappa}(2x)e^{p_1(x)}(2x)^{b/2}$,
 say
$f_{\kappa}(\lam)$, is of form
$f_{\kappa}(\lam) = p_{\kappa}(\lam) f_{0}(\lam)$
where $f_{0}(\lam)$   is the transform of the  function $e^{p_1(x)}(2x)^{b/2}$ and~$p_{\kappa}(\lam)$ is a symmetric polynomial which corresponds to a symmetric polynomial $p_{\kappa}(U_1, \dots, U_r)$ of the Cherednik operators under the Heckman--Opdam transform. In other words, there is a Rodrigues type formula expressing
 $L_{\kappa}(2x)e^{p_1(x)}(2x)^{b/2}$ as $p_{\kappa}(U_1, \dots, U_r)$ acting on
the simple weight function $e^{p_1(x)}(2x)^{b/2}$.
 However it might be more interesting to reverse the procedure, and to f\/ind the polynomials $p$ producing a Rodrigues type formula for $L_{\kappa}(2x)e^{p_1(x)}(2x)^{b/2}$, whose transform would then be an immediate consequence; see e.g.~\cite{pz-imrn, gz-imrn} for the case of Wilson polynomials.

\subsection*{Acknowledgements}
This work was done over a period of time when both authors were visiting the Newton Institute, Cambridge UK in July 2001, the Institute of Mathematical Sciences, Singapore National University in August 2002 and MPI/HIM (Bonn) July 2007. We would like to thank the institutes for their hospitality.
S.~Sahi was supported by a grant from the National Science Foundation (NSF) and
G.~Zhang by the Swedish Research Council (VR).

We dedicate this paper to the memory of our colleague and friend Tom Branson. We thank the f\/ive referees for helpful comments on an earlier version of this paper.

\pdfbookmark[1]{References}{ref}
\LastPageEnding

\end{document}